\documentclass{article}
\usepackage{amsmath,amsthm,amssymb}

\newtheorem{thm}{Theorem}[section]

\newtheorem{prop}[thm]{Proposition}

\begin{document}

\title{On a goodness of fit test for the Cauchy distribution}

\author{Emanuele Taufer}

\author{\emph{Emanuele Taufer} \\ Department of Economics and Management, University of Trento \\ \textsc{emanuele.taufer@unitn.it}
 }

\maketitle
\begin{abstract}
The paper discusses a test for the hypothesis that a random sample comes from the Cauchy distribution. The test statistics is derived from a characterization and is based on the characteristic function. Properties of the test are discussed and its performance measured by simulations. The test presented turns out to be extremely powerful in a wide range of alternatives.


\end{abstract}

\section{Introduction}
The Cauchy distribution with location $\theta \in \mathbb{R}$ and scale $\lambda>0$, denoted with $C(\theta, \lambda)$, has density
\begin{equation}
f(x; \theta, \lambda)= \frac{1}{\pi \lambda} \left[ 1+ \left(\frac{x-\theta}{\lambda}\right)^2 \right]^{-1}, \quad x \in \mathbb{R}
\end{equation}
and distribution function
\begin{equation}\label{DF}
F(x; \theta, \lambda)=\frac{1}{2} + \frac{1}{\pi} \arctan \left( \frac{x-\theta}{\lambda} \right). 
\end{equation}
The case where $\theta=0$ and $\lambda=1$ is referred to as the standard case.  Poisson was the first noting that the standard Cauchy distribution has some peculiar properties and could provide counterexamples to some generally accepted results in statistics; for an interesting historical account of this distribution see Stigler (1974). The Cauchy distribution has interesting applications in  seismography, chemistry and physics, see  for example  Kagan (1992), Winterton et al. (1992), Min et al. (1996), Stapf et al. (1996), Taufer et al. (2009); it also can serve as building block in statistical models  for volatility and multifractality, see, e.g. Anh et al. (2010), Meintanis and Taufer (2012), Leonenko et. al (2013) . For further applications and characterizations see Johnson et al. (2002).

Let $\phi(t)$ denote the characteristic function (CF) of a random variable X; from  Sato (1999, formula 13.1)  X is defined to be stable if for any $a>0$, there are $b>0$ and $c \in \mathbb{R}$ such that
\begin{equation}\label{FE}
\phi(t)^a=\phi(bt)e^{itc}
\end{equation}
The above relation has been exploited in Meintanis et al. (2015) to construct  goodness-of-fit tests for multivariate stable distributions. In this paper we aim at discussing in detail the case of the univariate Cauchy distribution which, as we will see, presents important peculiar characteristics and provides an extremely powerful goodness-of-fit test as it will be shown by comparisons with other tests presented in the literature. 

More formally, given a random sample, $X_1, \dots , X_n$  from some distribution $F$ we are interested in testing the hypothesis
\begin{equation}\label{Hypo}
H_0: F  \in {\cal C}= \{ C(\theta, \lambda): \theta \in \mathbb{R}, \lambda >0 \}.
\end{equation}

Previous related work closely connected to the approach followed here is that of Henze and Wagner (1997), G{\"u}rtler and Henze (2000), Matsui, M., \& Takemura, A. (2005); differently to those works, exploiting \eqref{FE} in the construction of the test statistics has several advantages: firstly one does  not need to specify a parametric form of the CF in the test statistics; secondly one can avoid demeaning the observations which means that if the scale parameter is specified by the null hypothesis there is no need to estimating parameters from the data; thirdly, the appropriate choice of  user-defined parameters allow to obtain extremely powerful tests for a large range of alternative; indeed, it will turn out that a specific  parametrization of the test statistic yields a powerful omnibus test for the Cauchy hypothesis.

The organization of the paper is the following: Section 2  shows that \eqref{FE} with the choice $a=b$ and $c=0$ actually characterizes the Cauchy distribution, the test statistics is introduced and previous contribution of the literature are discussed.  Section 3 presents a power study based on the Monte Carlo method and Section 4 concludes.

\section{Testing for the Cauchy hypothesis}

\subsection{A characterization of the Cauchy distribution}

The $C(\theta,\lambda)$ distribution has CF function
\begin{equation}
\phi(t) = \exp\{it \theta - \lambda|t|\}.
\end{equation}
In order to justify the test statistic proposed it is now proven formally that a specific choice of $a, b$ and $c$ characterizes the Cauchy distribution.

\begin{prop}
Formula \eqref{FE} holds for $b=a$ and $ c=0$ if and only if $\phi(t)$ is the CF of a $ C(\theta,\lambda)$ r.v..
\end{prop}

\begin{proof}
Suppose first that $\phi(t)$ is the CF of a $ C(\theta,\lambda)$, it is straightforward to verify that \eqref{FE} holds for for $b=a$ and $ c=0$.

Suppose now that \eqref{FE} is given, we will show that the choice $b=a$ and $ c=0$ implies that $\phi(t)$ is the CF of a $C(\theta,\lambda)$ random variable. Note first of all that, since for any $n>0$, $a/n>0$ and then $\phi(t)^{a/n}$ is a CF; it holds then that  $\phi(t)^a=[\phi(t)^{a/n}]^n$, and hence $\phi$ is an infinitely divisible CF. Next, from definition \eqref{FE} there exits $a_1, a_2>0$ with $a_1+ a_2=a$, $b_1, b_2>0$, $c_1, c_2 \in \mathbb{R}$ such that
\begin{equation}
\phi(t)^a=\phi(t)^{a_1}\phi(t)^{a_2}=\phi(tb_1)e^{itc_1}\phi(tb_2)e^{itc_2}
\end{equation}
from which we have 
\begin{equation}\label{LUC}
\phi(tb_1)\phi(tb_2)=\phi(bt)e^{itc'}
\end{equation}
with $c'=c-c_1-c_2$.
Formula \eqref{LUC}, implied by \eqref{FE}, corresponds to formula (5.7.2) in Luckacs (1970): there it is shown that the corresponding CF has the form (Luckacs, 1970, Theorem 5.7.2)

\begin{multline}\label{FCF}
\log \phi(t)=i t \mu -\frac{\sigma^2}{2} t^2 + \int_{-\infty}^0 \left(e^{itu} -1 - \frac{itu}{1+u^2}\right) d M(u)+ \\ + \int_0^{\infty} \left(e^{itu} -1 - \frac{itu}{1+u^2}\right) d N(u),
\end{multline}
where either $\sigma^2\neq 0$ and $M(u) \equiv 0$, $N(u) \equiv 0$ or $\sigma^2 = 0$ and $M(u) = K_1 |u|^{-\alpha} (u<0)$, $N(u) = -K_2 |u|^{-\alpha} (u>0)$ with $0<\alpha<2$, $K_1,K_2 \geq0$, $K_1+K_2>0$.

Consider first the case $M(u) \equiv 0$, $N(u) \equiv 0$, then relation \eqref{FE} does not hold for $b=a$ and $ c=0$ for the CF given by \eqref{FCF}.

Suppose now that $\sigma^2=0$, i.e. the stable case, we need to verify that equality \eqref{FE}  hold for $b=a$ and $ c=0$ only for the case $\alpha=1$. This is straightforward to see if we rewrite \eqref{FCF} in the more common form, for $0<\alpha<2$ (exclude for convenience the case $\alpha=2$), see Lukacs
Thm. 5.7.3
\begin{equation}
\log \phi(t) = i \theta t -\lambda|t|^\alpha \left[1+i \beta \frac{t}{|t|} \omega(|t|,\alpha) \right]
\end{equation}
where 
\begin{equation*}
\omega(|t|,\alpha)= 
\begin{cases} \tan(\pi \alpha/2) & \text{if $\alpha \neq 1$,}
\\
2 \pi \log|t| &\text{if $\alpha = 1$.}
\end{cases}
\end{equation*}
In our case, set $\beta=0$ and verify that \eqref{FE}  holds for $b=a$ and $ c=0$ only when $\alpha=1$.
\end{proof}

\subsection{Test statistic}

A test statistic for hypothesis \eqref{Hypo} can be defined by replacing $\varphi(\cdot)$ in \eqref{FE}  by an empirical estimate, say $\phi_n(\cdot)$, where 
\begin{equation}
\label{CF}
\phi_n(t)=\frac{1}{n} \sum_{j=1}^n e^{itX_j}
\end{equation}
and using the distance function
\begin{equation}\label{DeltaC}
\Delta_{n}(a,w)=n \int_{\mathbb {R}} |d_n(a,t)|^2 w(t) dt,
\end{equation}
with
\begin{equation}\label{CT}
d_n(a,t) = 	\phi_n(t)^a-\phi_n(t a)
\end{equation}
and where $w(\cdot)$ denotes a nonnegative weight function.  Recall that we can compute any power of a complex number by exploiting Euler's formula, i.e. for a complex number $z = u+ iv$, by Euler's formula we can write $z= |z| e^{i \theta}$ where $ \theta = \arctan (v/u)$ (in the appropriate quadrant) by which $z^a= |z|^a (\cos(a \theta)+ i \sin (a \theta))$.

The class of Cauchy distributions is closed under affine transformation, i.e. if $X \sim {\cal C}(\theta, \lambda)$ then, for constants $c,d$ it follows that $cX+d\sim  {\cal C}(c \theta+d, \lambda|c|)$; for this reason one is typically interested in affine invariant and consistent tests.  This can usually be achieved by transforming the data as $Y_j=(X_j-\hat \theta)/\sqrt{\hat \lambda}$, $j=1, \dots, n$ where $\hat \theta= \hat \theta(X_1, \dots, X_n)$ and  $\hat \lambda= \hat \lambda(X_1, \dots, X_n)$ are estimators of $\theta, \lambda$ such that, for each $c>0$, $d \in \mathbb{R}$,
\begin{equation}\begin{split}\label{AI}
\hat \theta(cX_1+d, \dots, c X_n+d)= c \hat \theta(X_1, \dots, X_n)+d \\
\hat \lambda(cX_1+d, \dots, c X_n+d)= c \hat \lambda(X_1, \dots, X_n)+d.
\end{split}\end{equation}
As it has been noted by Meintanis et al. (2015), the case of Cauchy distribution is quite peculiar:  if $\phi_n(t)$ is constructed using the variable $Y_j$, $j=1, \dots n$, then \eqref{CT} becomes
\begin{equation}
\phi_n(t)^a-\phi_n(t a)=e^{-i t a \hat \theta \hat \lambda^{-1/2}} [( \phi_n( \hat \lambda^{-1/2} t)^a- \phi_n( \hat \lambda^{-1/2} t \, a)].
\end{equation}
From the above equation one can see that $|d_n(a,t)|^2$ does not depend on the value of $\hat \theta $. For this reason based on Proposition 1 we can develop our test statistic using the scaled variables $Y_j=X_j/\sqrt{\hat \gamma}$ where$\hat \gamma$ is the   maximum likelihood (ML) estimate of $\gamma$ , i.e. the proposed estimator of $\phi(t)$ is
\begin{equation}\label{FTY}
\phi_n(t)=\frac{1}{n} \sum_{j=1}^n e^{itY_j}, \quad Y_j=X_j/\sqrt{\hat \gamma},
\end{equation}
with $\hat \lambda$  the maximum likelihood (ML) estimate of $\lambda$. The use of ML estimators has been suggested in the context of Cauchy testing by Matsui and Takemura (2005) which show their good performance also in comparison to other strategies of estimation. 

Finally, the choice $w(t)=\exp{[-\gamma t^2]}$, where $\gamma>0$ assures that \eqref{DeltaC} is affine invariant, its distribution does  not depend on $\theta, \lambda$ and explicit computational formulas are available.

In this final form the notation $\Delta_{n}(a,\gamma)$ is used for the test statistic  \eqref{DeltaC}. Implementation of the test requires to specify  the parameters $a$ and $\gamma$. It will be seen from the simulation results that a well specific combination of these values, namely $a=6$ and $\gamma=2.5$ provides an extremely powerful test for a large variety of alternatives.

From the results of Meintanis et al. (2015) one can derive consistency and the asymptotic distribution of $\Delta_{n}(a,\gamma)$. The asymptotic distribution can be expressed in terms of an infinite weighted sum of independent Chi-squared distributions. Given in our case we have a well defined indication in the choice of the parameters ($a=6$ and $\gamma=2.5$), for a practical implementation of the test, Table  \ref{tab:1} reports the critical points for tests of significance level $0.05$ and $0.1$.

For the case of integer parameter $a$, an explicit computational formula is available and it can be obtained by standard algebra as
\begin{equation}\begin{split}\label{CoFo}
\Delta_{n}(a,\gamma) =\frac{1}{n^{2a-1}} & \sum_{j_1, \dots j_a} I_w \left( Y_{j_1}+ \dots  Y_{j_a}- Y_{j_{a+1}}- \dots - Y_{j_{2a}}\right) +  \\
		& +  \frac{1}{n} \sum_{j_1, j_2} I_w \left( a(Y_{j_1} - Y_{j_{2}})\right) + \\
		& \quad -\frac{2}{n^{a}}  \sum_{j_1, \dots j_{a+1}} I_w \left( Y_{j_1}+ \dots  Y_{j_a} - a Y_{j_{a+1}}\right) 
\end{split}\end{equation}
where $Y_j$, $j=1, \dots n$ has been defined in \eqref{FTY} and summation is over the indexes $j_k =1, \dots, n$, $k=1,  \dots 2a$. Also,
\begin{equation}
I_w(x) = \int_{\mathbb R} \cos (tx) w(t) \, dt = \sqrt{\frac{\pi}{\gamma}} \exp\{- \frac{x^2}{4 \gamma} \}
\end{equation}
when $w(t)=\exp{[-\gamma t^2]}$. Formula \eqref{CoFo} is of order $n^{2a}$ and might become cumbersome to implement if the value $a$ is too large. In such a case one can resort to numerical approximations. 

\begin{table}[]
\renewcommand{\arraystretch}{1}
\centering
\caption{\footnotesize  Critical values the test statistics $\Delta_{n}(a,\gamma)$ with $a=6$ and $\gamma=2.5$ for $0.1$ and $0.05$ significance level tests.}
	\label{tab:1}
\begingroup\setlength{\fboxsep}{0pt}
		\begin{tabular}{ll|rrrr}
		\hline\hline		
			& $n \rightarrow$ & 10 & 30 & 50 & 100 \\ \hline	         
	Sig.	&0.05  & 2.61 & 2.99 & 3.00 & 3.00 \\ 
		& 0.10   & 2.29 & 2.56 & 2.59 & 2.59\\
				          \hline

\end{tabular}
\endgroup
\end{table}

\subsection{Other tests for the Cauchy hypothesis }

Standard classical tests such as the Kolmogorov-Smirnov, the Cram\'{e}r-von Mises, the Anderson Darling and the Watson statistics can be applied to test for \eqref{Hypo} by comparing the estimated empirical distribution function of $Y_j=1/2-\pi^{-1} \arctan[(X_j-\hat \theta)/\hat \lambda]$, $j=1, \dots, n$ and the distribution function of a $U(0,1)$. Details can be found in D'Agostino and Stephens (1986).

Among the mainstream approaches to test hypothesis \eqref{Hypo}, we find those proposed by  G{\"u}rtler and Henze (2000), Matsui and Takemura (2005) which compare  $\exp\{-|t|\}$ with the empirical CF \eqref{CF} calculated on the standardized data $Y_j=(X_j-\hat \theta)/\hat \lambda$, $j=1, \dots, n$. These approaches use a distance function analogous to \eqref{DeltaC} and differentiate each other for the estimation procedure adopted for estimation of $ \theta,  \lambda$. 

Meintanis (2001) has developed two test statistics for hypothesis \eqref{Hypo}: the first test exploits the fact that, under $H_0$, $c(t)$ and $\delta(t) $ are constant $\forall t \neq 0$ and where $c(t) =|t|^{-1} \log |\phi(t)|$ and $\delta(t)=t^{-1} \arg \phi(t)$ with $\arg \phi(t) = \tan^{-1} (I(t)/R(t))$ and $I(t)$ and $R(t)$ are, respectively, the imaginary and real part of $\phi(t)$. A second test is developed exploiting the functional equation  $|\phi(t+s)|^2 = |\phi(t)|^2|\phi(s)|^2$.  G{\"u}rtler and Henze (2000) note that these test statistics, while they are free from standardization and have simple asymptotic null distribution, are consistent only against some subclasses of alternatives.

Another available test, based on quantiles, is that of Rublík (2001), which defines a  test statistics based on the differences
\begin{equation}
\Delta_n= \left( F(X_{(1)}, \hat \theta, \hat \lambda)-\frac{1}{n+1}, F(X_{(n)}, \hat \theta, \hat \lambda)-\frac{n}{n+1}\right)
\end{equation}
where $F$ is defined in \eqref{DF} and $X_{(i)}$ $i=1, \dots n$ indicates the order statistics of the sample. The asymptotic distribution of this statistic is unknown, however it does not depend on the parameters of the underlying distribution.

A further proposal based on characterization is that of Litvinova (2005), which exploits sample counterpart of the following characterizations: a) given $X,Y$ continuous and independent random variables, then $X \stackrel{D}{=} (X+Y)/(1-XY)$ if and only if $X$ and $Y$ are ${\cal C}(0,1)$; b) given $X,Y$ continuous and independent random variables and given constants $|a|+|b|=1$ such that $-\log a$ and $-\log b$ are incommensurable then $X \stackrel{D}{=} aX+bX$ if and only if $X$ is ${\cal C}(0,\lambda)$ where the scale parameter is arbitrary. The test statistics are based on $U$-statistics counterparts of the above characterization.  For further review and comparisons see also Onen et al. (2001).

\section{Monte Carlo comparisons}

We perform here a simulation study to investigate the actual performance of the test statistic under various alternatives. The chosen alternatives, sample sizes and significance level allow us to compare the performance of test statistics \eqref{DeltaC} with the extensive simulation of G{\"u}rtler and Henze (2000) and Matsui and Takemura (2005) which develop CF-based tests and compare their performance to  classical tests such as the Kolmogorov-Smirnov, the Cram\'{e}r-von Mises, the Anderson Darling and the Watson statistics (see above for details), a powerful  UMP invariant test for hypothesis \eqref{Hypo} against normality is also discussed in G{\"u}rtler and Henze (2000).

\begin{table}[]
\renewcommand{\arraystretch}{1}
\centering
\caption{\footnotesize Percentage of rejection of the test for the univariate Cauchy null hypothesis; Sample size $n=20$, $a=2, 4, 6$, $\gamma = 0.5, 1, 2.5$; Significance level $q=10\%$, $M=3000$ Monte Carlo trials.}
	\label{tab:2}
\begingroup\setlength{\fboxsep}{0pt}
		\begin{tabular}{l|rrr|rrr|rrr}
		\hline\hline		
				              $\gamma  \rightarrow$          & \multicolumn{3}{c}{0.5}&\multicolumn{3}{c}{1}& \multicolumn{3}{c}{2.5} \\ \hline
		  		    $a \rightarrow$  & {\it 2} &  {\it 4} &  {\it 6}& {\it 2} &  {\it 4} &  {\it 6} & {\it 2} &  {\it 4} &  {\it 6} \\ \hline
	   {$t_1$}         &9&11&9&9&10&9&9&10&9\\
	    {$t_2$}     &12&9&6&19&20&21&25&24&27 \\
	{$t_4$}     &15&15&9&34&38&41&44&46&50 \\    
	   	{$t_5$}     &16&17&9&38&42&44&49&51&55 \\
	   	{$t_{10}$}   &21&23&15&47&52&56&59&61&68 \\ \hline
	   	{${ {S}}_{0.5}$}  &29&34&42&32&30&33&34&30&30  \\
	   	{${{ S}}_{0.8}$}   &13&16&16&11&9&11&9&6&9 \\
	   	{${{S}}_{1.2}$}  &10&10&6&14&14&13&16&16&17  \\
	   	{${{S}}_{1.5}$}  &14&15&7&27&31&30&34&36&37  \\
	   	{${{S}}_{1.7}$} &16&18&11&38&41&44&47&49&53 \\ \hline
	   	$Tuk_{1.0}$  &11&12&12&9&8&9&9&7&8 \\
	   	$Tuk_{0.2}$  &13&14&7&30&32&34&39&40&44 \\
	   	$Tuk_{0.1}$  &19&18&11&42&43&49&54&53&60 \\
	   	$Tuk_{0.05}$  &21&22&14&50&52&56&61&63&67 \\ \hline
	   	{${{N}}(0,1)$}    &26&29&19&56&62&66&68&71&75 \\
	   	 {$Lap$}    &12&12&7&24&26&28&32&32&37  \\
	   	 {${{U}}(0,1)$}    &72&78&74&90&93&95&93&94&96  \\ \hline

\end{tabular}
\endgroup
\end{table}

The alternative distributions considered here are: 
\begin{itemize}

\item[1)] symmetric $t$-Student distributions, denoted with $t_{\nu}$ where $\nu$ are the degrees of freedom.

\item[2)] Symmetric $\alpha$-stable distributions indicated with $S_{\alpha}=S(\alpha,0,0,1)$ where $S(\alpha, \beta, \mu, \sigma)$  denotes a stable distribution with index of stability $\alpha$, and where $\beta$, $\mu$ and $\sigma$ indicate, respectively, asymmetry, location and  scale; here we have $0< \alpha \leq 2$, $0 \leq \beta \leq 1$, $\mu \in \mathbb{R}$, $\sigma>0$.

\item[3)] Tukey distributions, denoted with $Tuk_\nu$; where $\nu$ is a parameter used in the transformation $Y = Z \exp\{ \nu Z^2/2\}$ where $Z$ is a standard normal random variable and $Y$ has Tukey distribution with parameter $\nu$. The case $\nu=0$ corresponds to the standard normal distribution, while the case $\nu=1$ gets a Cauchy-like distribution. 

\item[4)] Other short-tailed and long tailed alternatives, such as the standard normal distribution, denoted as $N(0, 1)$, the Laplace distribution with density $f(x)=\exp\{-|x|\}/2$, $x \in \mathbb{R}$, denoted simply with $Lap$ and  the uniform on the unit interval, denoted by $U(0,1)$.
\end{itemize}

\begin{table}[]
\renewcommand{\arraystretch}{1}
\centering
\caption{\footnotesize Percentage of rejection of the test for the univariate Cauchy null hypothesis; Sample size $n=50$, $a=2, 4, 6$, $\gamma = 0.5, 1, 2.5$; Significance level $q=10\%$, $M=3000$ Monte Carlo trials.}
	\label{tab:3}
\begingroup\setlength{\fboxsep}{0pt}
		\begin{tabular}{l|rrr|rrr|rrr}
		\hline\hline		
				              $\gamma  \rightarrow$          & \multicolumn{3}{c}{0.5}&\multicolumn{3}{c}{1}& \multicolumn{3}{c}{2.5} \\ \hline
		  		    $a \rightarrow$  & {\it 2} &  {\it 4} &  {\it 6}& {\it 2} &  {\it 4} &  {\it 6} & {\it 2} &  {\it 4} &  {\it 6} \\ \hline
	   {$t_1$}        & 10	 & 10 & 10 &  10 & 10 & 10 &10&9&10\\
	    {$t_2$}     & 16  &36&47& 16  &41&52  &16&43&54 \\
			{$t_4$}     &  35 &74&87&  41&79&89 &45&84&92 \\    
	   	{$t_5$}     &  41 &82&91& 48 &87&94  &54&89&96 \\
	   	{$t_{10}$}    &  55 &92&98& 66 &95&98  &73&97& 99 \\ \hline
	   	{${ {S}}_{0.5}$}  &  74 &84& 89 &78 &84& 90  &83&87&89  \\
	   	{${{ S}}_{0.8}$}   &  19   &16 &19 & 21 &16 & 21 &28&22&21 \\
	   	{${{S}}_{1.2}$}  & 12  &16 &23 & 11 & 16 &25 &9&20&25  \\
	   	{${{S}}_{1.5}$}  & 26  &50 &64 & 29 & 51 & 66 &28&60&68  \\
	   	{${{S}}_{1.7}$}  & 43  &75 &86 & 49 & 76& 88 &53&85&89 \\ \hline
	   	$Tuk_{1.0}$  &11&13&14&11&13&11&10&11&12 \\
	   	$Tuk_{0.2}$  &28&32&34&63&68&73&77&80&85 \\
	   	$Tuk_{0.1}$  &43&52&60&84&89&92&77&80&85 \\
	   	$Tuk_{0.05}$  &59&67&77&94&96&97&98&98&99 \\ \hline
	   	{${{N}}(0,1)$}   & 72  & 96 & 99 & 81 & 97 & *  & 89 &99&* \\ 
               {$Lap$}    &20&23&29&52&58&66&70&74&81  \\
	   	 {${{U}}(0,1)$}    &* &* &* &* &* &* &* &* &*   \\ \hline
\end{tabular}
\endgroup
\end{table}

Table \ref{tab:2} and Table \ref{tab:3} report, respectively for $n=20$ and $n=50$ the power  of $\Delta_{n}(a,\gamma)$ expressed as the percentage of samples declared significant for the alternative considered; results are based on 3000 Monte Carlo replications in each case. The results give some very clear indication which we summarize below:

\begin{itemize}

\item The choice of $a$ and with $\gamma$ has a strong impact on the power of the test; however the case   $a=6$ and $\gamma=2.5$ is always the one with highest or close to highest power. This results suggest using this case as the standard one in applications. Further simulation results, not shown here, confirm that a too large valure of $\gamma$ results in a loss of power and shall not be considered.

\item Comparing the results of the table with those appearing in Tables 6 and 7 of G{\"u}rtler and Henze (2000) and Tables 5 and 6 of Matsui and Takemura (2005), which have the same set-up, we see that the choice of $a$ and $\gamma$ suggested, in  all cases, but the $S_\alpha$ with $\alpha<1$, has higher, or comparable, power with respect to the CF-based and other test statistics considered in those papers. The power of the test based on $\Delta_n(6,2.5)$ in most cases is also comparable with the power of the UMP invariant test against normality discussed in G{\"u}rtler and Henze (2000).

\end{itemize}

\section{Conclusions}

A test based on a characterization of the Cauchy distribution has been discussed. The test statistic presented does not depend on location, is consistent, its distribution is parameter free and has extremely high power, compared to several competitor tests, in a wide range of alternative distributions. From practical point of view the test is quite simple to implement  given  the choice  $a=6$ and $\gamma=2.5$ always performs well. Critical values for this case are available and stabilize already for $n=50$.

\end{document}